\documentclass[11pt]{article}

\usepackage[backend=biber,style=alphabetic]{biblatex}
\usepackage[utf8]{inputenc}
\usepackage[T1]{fontenc}

\usepackage[english]{babel}
\usepackage{amsthm,amssymb,amsmath,amsfonts,mathrsfs,amscd,stmaryrd}
\usepackage{verbatim}
\usepackage{dsfont}
\usepackage[letterpaper,top=2cm,bottom=2cm,left=3cm,right=3cm,marginparwidth=1.75cm]{geometry}
\usepackage{comment}
\usepackage{extarrows}

\usepackage{url}
\usepackage{amsmath,mathtools}
\usepackage{graphicx}
\usepackage[colorlinks=true, allcolors=blue]{hyperref}
\usepackage{tikz-cd}
\usepackage{bbold}

\newcommand{\C}{\mathbb{C}}
\renewcommand{\O}{\mathcal{O}}
\newcommand{\K}{\mathcal{K}}
\newcommand{\Z}{\mathbb{Z}}

\renewcommand{\S}{\mathcal{S}}
\newcommand{\B}{\mathcal{B}}

\newcommand{\R}{\mathcal{R}}

\renewcommand{\H}{\textbf{H}}
\newcommand{\Hsph}{\textbf{H}^\text{sph}}

\renewcommand{\P}{\mathbb{P}}
\newcommand{\g}{\mathfrak{g}}

\renewcommand{\t}{\mathfrak{t}}

\newcommand{\tG}{\mathbb{G}}
\newcommand{\bbG}{\mathbb{G}}
\newcommand{\bbB}{\mathbb{B}}

\newcommand{\X}{\mathbb{X}}
\renewcommand{\th}{\mathbb{h}}
\newcommand{\bbh}{\mathbb{h}}
\newcommand{\tW}{\mathbb{W}}
\newcommand{\tT}{\mathbb{T}}
\newcommand{\tB}{\mathbb{B}}
\newcommand{\tf}{\widetilde{f}}
\newcommand{\tPi}{\widetilde{\Pi}}

\newcommand{\HBM}[2]{H_\bullet^{#1}(#2)}

\newcommand{\Crot}{\C^*_{\text{rot}}}
\newcommand{\Cdil}{\C^*_{\text{dil}}}

\newcommand{\tM}{\widetilde{M}}

\newcommand{\Fl}{\mathcal{F}l}

\newcommand{\Hom}{\text{Hom}}
\newcommand{\End}{\text{End}}

\newcommand{\Res}{\text{Res}}

\newcommand{\Waff}{W_{\text{aff}}}

\newcommand{\pt}{\text{pt}}

\newcommand{\bbPhire}{\mathbb{\Phi}^\text{re}}
\newcommand{\ddH}{\overset{\bullet\bullet}{\mathrm{H}}}

\newcommand{\sw}{\operatorname{sw}}

\renewcommand{\subset}{\subseteq}

\renewcommand{\Im}{\text{Im}}

\newtheorem{theorem}{Theorem}[section]
\newtheorem{lemma}[theorem]{Lemma}
\newtheorem{proposition}[theorem]{Proposition}

\newtheorem{remark}[theorem]{Remark}
\newtheorem{corollary}[theorem]{Corollary}

\theoremstyle{definition}
\newtheorem{definition}[theorem]{Definition}
\newtheorem{definition-lemma}[theorem]{Definition-Lemma}

\numberwithin{equation}{section}

\begin{document}

\title{Convolution algebras associated to representations}

\date{}
\author{Dragoș Crișan \\
  \small Department of Mathematics, University of Chicago \\
  \small \texttt{dcrisan@uchicago.edu}}

\maketitle

\begin{abstract}
Given a complex reductive group $G$, a representation $V$ of $G$ and a Borel-stable subspace $M \subset V$, we consider the associated Steinberg-type variety $Z$. We prove that, under a certain condition on $(V,M)$, called gluability, the equivariant Borel-Moore homology or $K$-theory of $Z$, equipped with the convolution product, is obtained as the intersection of two copies of the nil-Hecke algebra inside its localization. We also provide a description of these new algebras in terms of poles and residues. Similar results are obtained when $G$ is replaced by its loop group. This generalizes results of Ginzburg, Kapranov and Vasserot describing the affine Hecke algebra and DAHA, as well as a result of Teleman and Gannon--Webster that realizes certain Coulomb branches by gluing two copies of the universal centralizer.
\end{abstract}

\section{Introduction}
    
Let $G$ be a complex reductive group, with a simply connected derived subgroup. Fix a maximal torus and a Borel subgroup $T\subset B$ and let $W$ be the Weyl group. Let $\B \cong G/B$ be the flag variety of $G$. Let $S=\C[\t][\hbar]$, $Q=\text{Frac}(S)$ and $Q[W] = Q \rtimes \C[W]$. The usual action of $W$ on $\t^*$ extends to an action of $Q[W]$ on $Q$. Let $\H$ be the subalgebra of operators in $Q[W]$ which preserve $S\subset Q$. Proposition 2.3 in \cite{GKV} shows that $\H \cong \End_{S^W}(S)$. The algebra $\H$ is a version of the nil-Hecke algebra, and we have an algebra isomorphism $\H \cong \HBM{G\times \C^*}{\B \times \B}$, where $\C^*$ acts trivially.

Let $V$ be a finite dimensional representation of $G$ and let $M\subset V$ be a $B$-stable subspace. We call the pair $(V,M)$ gluable if, for all $w\in W$, the $T$-representations $M/(M \cap wM)$ and $wM / (M \cap wM)$ have no common weights. A more general definition of gluability is discussed in Definition \ref{finite_gluability_def} and Remark \ref{finite_gluability_remark}. This assumption holds, for example, when the intersection of $M$ with any weight space of $V$ is either trivial or the full weight space. Let $\Pi_M \in S$ be the product of weights of the $T\times \C^*$-representation $M$, counted with multiplicity. Here, $\C^*$ acts by dilations on $V$.

Define the vector bundle $\tM = G\times^B M \rightarrow G/B$ and the Steinberg-type variety given by $Z = \tM \times_V \tM$. Consider the equivariant Borel-Moore homology $\HBM{G\times \C^*}{Z}$, equipped with the convolution product. The main result of this paper is:

\begin{theorem}[Theorem \ref{finite_H(Z)=H_M}]
\label{intro_theorem}
If the pair $(V,M)$ is gluable, then there exists an isomorphism of algebras
    $$\HBM{G\times \C^*}{Z} \cong \H \cap \Pi_M^{-1} \H \Pi_M$$
\end{theorem}

Here $\Pi_M^{-1}  \H \Pi_M \coloneqq \{ \frac{1}{\Pi_M} \cdot g \cdot \Pi_M \in Q[W]  \mid g \in \H\}$ and the intersection is computed in $Q[W]$. We also give a poles and residues construction of the algebra $\H \cap \Pi_M^{-1}  \H \Pi_M$ in the spirit of \cite{GKV}. Similar results hold when replacing Borel-Moore homology with $K$-theory.

Denote by $\tau_M$ the anti-involution of $\HBM{G\times \C^*}{Z}$ induced by the map $(x,y)\mapsto (y,x)$ in $Z\subset \tM \times \tM$. We also construct geometrically an $S^W$-valued bilinear form on $S$ with respect to which $\tau_M$ is the adjoint anti-involution. Extending $\tau_M$ to an anti-involution of $Q[W]$, we can reformulate Theorem \ref{intro_theorem} to obtain $\HBM{G\times \C^*}{Z} \cong \H \cap \tau_M(\H)$. 

When $V$ is the adjoint representation of $G$ and $M$ is the Lie algebra of the unipotent radical of $B$, the variety $Z$ is the usual Steinberg variety. In this case, combining Theorem \ref{intro_theorem} with results of Ginzburg, Kazhdan and Lusztig \cite{KL, Gi2, Lu}, we recover a result of \cite{GKV}, which realizes the affine Hecke algebra and its degeneration as the intersection of two copies of the (degenerate) nil-Hecke algebra.

We also obtain a version of Theorem \ref{intro_theorem} when the group $G$ is replaced by its loop group, and $B$ is replaced by an Iwahori subgroup or the arc group. The spherical version is a quantized version of the gluability construction of Coulomb branches, due to Teleman and Gannon--Webster \cite{Tel, GW}, and the Iwahori version can be seen as a gluing construction for quantized Iwahori-Coulomb branches. The $K$-theoretic version of the latter also recovers descriptions of DAHA in \cite{GKV, BEG}.

The current paper is organized as follows. In Section \ref{section_algebra}, the algebra $\H \cap \Pi^{-1} \H \Pi$ is defined and a poles and residues construction of it is provided (Theorem \ref{H_Sigma_(1)(2)(3)}). In Section \ref{section_finite} we study the geometry of the Steinberg-type variety $Z$ and we prove a slightly more general version of Theorem \ref{intro_theorem}. We also discuss the associated bilinear form and the anti-involution $\tau_M$. In Section \ref{section_affine} we prove Iwahori and spherical analogues of Theorem \ref{intro_theorem} in the affine case.

\subsection*{Acknowledgements}
I would like to thank my advisor, Victor Ginzburg, for suggesting this problem, for directing me towards the relevant literature and for many helpful discussions. I am also grateful to Michael Finkelberg and Pavel Etingof for very interesting suggestions and for pointing out a connection to \cite{BEF}. I also thank Daniil Klyuev for informing me about his forthcoming paper \cite{Kly2} and for communicating several related computations.

The author used ChatGPT (Sol 5.6, OpenAI) to review an earlier version of this paper. In particular, ChatGPT identified the need for condition (2b) in Theorem \ref{H_Sigma_(1)(2)(3)}. The entire paper was written and independently verified by the author, who takes full responsibility for its contents.

\section{Algebra}

\label{section_algebra}

\subsection*{Preliminaries on nil-Hecke algebras}

We use the standard framework of Kac-Moody algebras, as explained in \cite{Ka} or chapter I of \cite{Ku}. Let $A=(a_{ij})$ be a symmetrizable generalized Cartan matrix. Fix a root datum associated to $A$, that is $({\X}, {\X}^\vee, \mathbb{\Pi}, {\mathbb{\Pi}}^\vee)$, where ${\X}$ is a free abelian group, ${\X}^\vee \coloneqq \Hom_\Z ({\X},\Z)$, $\mathbb{\Pi}=\{ \alpha_1, \alpha_2, \ldots \alpha_n\} \subset \X$ is the set of simple roots, $\mathbb{\Pi}^\vee = \{ \alpha_1^\vee, \alpha_2^\vee, \ldots \alpha_n^\vee\}$ is the set of simple coroots. The elements of $\mathbb{\Pi}$ and of $\mathbb{\Pi}^\vee$ are required to be linearly independent and to satisfy the relation $\langle \alpha_i^\vee, \alpha_j \rangle = a_{ij}$, where $\langle-,-\rangle : \X^\vee \times \X \rightarrow \Z$ is the usual pairing. Let $\mathbb{h}^* \coloneqq \C \otimes_{\Z}{\X}$ be the dual Cartan and $\mathbb{h} \coloneqq \C\otimes_{\Z} \X^\vee$ be the Cartan. The triple $(\mathbb{h},{\mathbb{\Pi}},{\mathbb{\Pi}}^\vee)$ is called a realization of $A$. Associated to this realization, one defines the Kac-Moody algebra $\mathbb{g}$. Associated to the root datum, one defines the Kac-Moody group $\tG$, its Borel subgroup $\tB$ and its maximal torus $\tT$. Let $\mathbb{\Phi}\subset \mathbb{h}^*$ be the set of roots of the realization, let $\mathbb{\Phi}_\pm$ be the set of positive/negative roots, let $\bbPhire$ be the set of real roots and let $\bbPhire_\pm \coloneqq \bbPhire \cap \mathbb{\Phi}_\pm$ be the set of positive/negative real roots.  Let $\tW$ be the associated Weyl group, and let "$\leq$" be the Bruhat order on it.

We work with two different setups in parallel, to highlight the similarities and differences between them. We think of the first setup as being "in homology", while the second one is "in $K$-theory". This will be later justified geometrically. We use cursive letters to denote the $K$-theoretic analogue of the objects in homology denoted by roman or greek letters. Although we only write the results in homology, they hold with the same proofs in $K$-theory.

We recall the definition of the nil-Hecke algebra, introduced in \cite{KK1}, see also chapter XI of \cite{Ku}. Let $S = \C[\bbh]$, $Q=\C(\bbh)$. Let $Q[\tW] \coloneqq Q \otimes_\C \C[\tW]$. To simplify notation, denote $f[w] \coloneqq f\otimes w$. Define a multiplication on $Q[\tW]$ by:
$$f_1 [w_1]\cdot f_2 [w_2] \coloneqq f_1 \cdot w_1(f_2) [w_1 w_2]$$
\noindent where $w_1(-)$ is the extension to $Q$ of the action of $w_1 \in \tW$ on $\bbh$. Identify $Q$ with $Q[1]$ in $Q[\tW]$. We let $Q[\tW]$ act on $Q$, where $\tW$ acts as above and $Q$ acts by left multiplication. We denote this action by "." .

For every simple root $\alpha_i \in \mathbb{\Pi}$, define
$$D_{s_i} \coloneqq -\frac{1}{\alpha_i} [s_{\alpha_i}] + \frac{1}{\alpha_i} [1] \in Q[\tW]$$

The elements $D_{s_i}$ satisfy the braid relations in $\tW$, so if $w = s_{i_1} s_{i_2} \ldots s_{i_\ell}$ is a reduced expression for $w\in \tW$, we can define 
$$D_w \coloneqq D_{s_{i_1}} D_{s_{i_2}} \ldots D_{s_{i_\ell}}$$
\noindent independent of the choice of reduced expression. These are called the Demazure operators and they satisfy $D_{s_i}^2=0$. Let $\H$ be the subring of $Q[\tW]$ generated by $\{D_w\}_{w\in \tW}$ and $S$. This is called the nil-Hecke algebra.

\begin{proposition}
\label{tH_description}
    $\H$ is a free left $S$-module, with a basis given by $\{D_w\}_{w\in \tW}$.\\
    $\H = \{ f\in Q[\tW] \mid f.S \subset S \}$    
\end{proposition}

When $A$ is of affine type, we extend the definition of the nil-Hecke algebra. Let $\K=\C((\varpi))$ and $\O=\C[[\varpi]]$. Let $G$ be a finite dimensional reductive group with simply connected derived subgroup, let $T \subset B$ be a maximal torus and a Borel subgroup and let $W$ be the Weyl group. Let $G(\K)$ be the corresponding loop group and consider its extension $\mathbb{G} = G(\K) \rtimes \Crot$, where $\Crot \cong \C^*$ acts on $G(\K)$ by loop rotations. For our purposes, we ignore the central extension by $\C^*$, as it will always act trivially. Let $\mathbb{B} = I \rtimes \Crot$, $\mathbb{T} = T \times \Crot$, where $I$ is the Iwahori subgroup. Then $\mathbb{h} = \operatorname{Lie} \mathbb{T}$ is the affine Cartan subalgebra and $\mathbb{\Phi}$ is the set of affine roots. The affine Weyl group $\Waff$ is the Coxeter group generated by the simple affine reflections. Let $S=\C[\th]$ and let $Q$ be the ring of rational functions on $\th$ with no poles along imaginary root hyperplanes. The same formulas as above define Demazure operators $D_w \in Q[\Waff]$, for $w \in \Waff$. We extend this to the extended affine Weyl group $\tW \coloneqq X_*(T) \rtimes W$ as follows. The length function on the Coxeter group $\Waff$ extends to a length function $\ell$ on $\tW$. Let $\Omega = \{ w\in \tW \mid \ell(w)=0\}$ be the stabilizer in $\tW$ of the fundamental alcove. Then $\tW = \Waff \rtimes \Omega$. For $w = a\omega \in \tW$, with $a\in \Waff$ and $\omega \in \Omega$, define 
$$D_w \coloneqq D_a [\omega] \in Q[\tW]$$

As before, define the (extended) nil-Hecke algebra $\H$ as the subring of $Q[\tW]$ generated by $\{ D_w\}_{w\in \tW}$ and $S$. Proposition \ref{tH_description} still holds for this extended version of $\H$, and all of the results that follow will hold with the same proofs. We define $\H$ in a similar way for $\bbG = G(\K) \rtimes \Crot$ or $\bbG = G(\K) \rtimes \Crot \times \Cdil$, where $\Cdil=\C^*$.

Following \cite{KK2}, assuming the initial root datum is simply connected, we define the $K$-theoretic analogues $\S \coloneqq \C[\tT]$, $\mathcal{Q}\coloneqq \C(\tT)$ (or rational functions on $\tT$ with no poles at imaginary roots, in the affine case). We denote elements of $X^*(\tT)$ by $e^{\alpha}$. The $K$-theoretic Demazure operators for simple reflections are defined by
$${\mathcal{D}_{s_i}} \coloneqq - \frac{e^{-\alpha_i}}{1-e^{-\alpha_i}} [s_{\alpha_i}]+ \frac{1}{1-e^{-\alpha_i}} [1] \in \mathcal{Q}[\tW]$$

They satisfy the braid relations, so one defines $\mathcal{D}_w$ for any $w\in \tW$ as before, but $\mathcal{D}_{s_i}^2 = {\mathcal{D}_{s_i}}$. As in homology, define the $K$-theoretic nil-Hecke algebra $\mathcal{H}$ to be the subring of $\mathcal{Q}[\tW]$ generated by $\mathcal{S}$ and $\{ 
\mathcal{D}_w\}_{w\in \tW}$. In the affine case, we define the extended nil-Hecke algebra (nil-DAHA) using the same technique as in homology. This is constructed in \cite{Hai}, \cite{CF}.

\subsection*{Definition and properties of $\H_\Sigma$}

Let $X^*(\tT)$ be the character lattice of $\tT$. Note that $\tW$ acts on it. Let $\Sigma$ be a collection of elements of $X^*(\tT)$, in which every $\lambda \in X^*(\tT)$ appears $m_\lambda$ times, where $m_\lambda \in \Z_{\geq 0}$. 

\begin{equation}
\label{finiteness_assumption}
\tag{A}
\parbox{0.85\textwidth}{We assume that for all $w\in \tW$, the collections $w(\Sigma) \setminus \Sigma$ and $\Sigma \setminus w(\Sigma)$ are finite (i.e. only finitely many weights appear with non-zero multiplicity).}
\end{equation}

\noindent Define
$$\Pi_\Sigma \coloneqq \prod_{\lambda \in X^*(\tT)} \lambda^{m_\lambda}$$

\noindent Its $K$-theoretic analogue is
$$\varPi_\Sigma\coloneqq \prod_{\lambda \in X^*(\tT)} (1-e^{-\lambda})^{m_\lambda}$$

\noindent In general, $\Pi_\Sigma$ is an infinite product, so it is not an element of $Q$. We will treat it as a formal symbol. Let $w\in \tW$. Assumption (\ref{finiteness_assumption}) ensures that $\frac{\Pi_\Sigma}{\Pi_{w(\Sigma)}}$ is a well-defined rational function on $\th$. Expressing this fraction in lowest terms, let

$$\frac{\Pi_\Sigma}{\Pi_{w(\Sigma)}} = \frac{\tPi_{\Sigma, w(\Sigma)}}{\tPi_{w(\Sigma), \Sigma}}$$

\noindent where $\tPi_{\Sigma, w(\Sigma)}$ and  $\tPi_{w(\Sigma), \Sigma}$ are coprime elements of $S$. 

\noindent When $\Sigma$ is finite, $\tPi_{\Sigma, w(\Sigma)} = \frac{\Pi_\Sigma}{\gcd(\Pi_\Sigma, \Pi_{w(\Sigma)})}$ and $\tPi_{w(\Sigma), \Sigma} = \frac{\Pi_{w(\Sigma)}}{\gcd(\Pi_\Sigma, \Pi_{w(\Sigma)})}$.

Assumption (\ref{finiteness_assumption}) also allows us to define an operator $\Pi_\Sigma^{-1} (-) \Pi_\Sigma : Q[\tW] \rightarrow Q[\tW]$ by:

$$\Pi_\Sigma^{-1} \left(\sum_{w\in \tW} f_w [w]\right) \Pi_\Sigma \coloneqq \sum_{w\in \tW} f_w \frac{\Pi_{w(\Sigma)}}{\Pi_\Sigma}[w]$$

\begin{definition}
    Define the subalgebra $\H_\Sigma \subset Q[\tW]$ by $\H_\Sigma \coloneqq \H \cap \Pi_\Sigma^{-1} \H \Pi_\Sigma$.
\end{definition}

The goal of this section is to study the algebra $\H_\Sigma$. We describe it as the stabilizer in $\H$ of an $S$-module in Proposition \ref{H_Sigma_is_normalizer}. We give a poles and residues description of it in Theorem \ref{H_Sigma_(1)(2)(3)}, similar to \cite{GKV}. In a special case, we also describe a basis of $\H_\Sigma$ viewed as an $S$-module in Proposition \ref{H_Sigma_basis}.

Let $Q \Pi_\Sigma^{-1}$ be a free rank $1$ module over $Q$, with a generator denoted by $\Pi_\Sigma^{-1}$. Define a left $Q[\tW]$-module structure on it by linearly extending:
$$f_w [w] . (g \Pi_\Sigma^{-1}) \coloneqq f_w w(g) \frac{\Pi_\Sigma}{\Pi_{w(\Sigma)}} \Pi_\Sigma^{-1}$$

\noindent Let $S \Pi_\Sigma^{-1}$ be the $S$-submodule of $Q \Pi_{\Sigma}^{-1}$ generated by $\Pi_{\Sigma}^{-1}$. This is a free $S$-module of rank $1$.

\begin{proposition}
\label{H_Sigma_is_normalizer}
We have equalities:
$$\H_\Sigma = \{ f \in \H \mid f.(S \Pi_\Sigma^{-1}) \subset S \Pi_\Sigma^{-1} \} = \{ f \in Q[\tW] \mid f.S \subset S \text{ and } f.(S \Pi_\Sigma^{-1}) \subset S \Pi_\Sigma^{-1}\}$$
\end{proposition}

\begin{proof}
    The second equality follows from Proposition \ref{tH_description}. It follows from the definitions that $f.(S \Pi_{\Sigma}^{-1}) \subset S \Pi_\Sigma^{-1}$ if and only if $(\Pi_\Sigma f \Pi_{\Sigma}^{-1}).S \subset S$. Then, the first equality follows from Proposition \ref{tH_description} and the definition of $\H_\Sigma$.
\end{proof}

For $\alpha \in X^*(\tT)$, let $\mathbb{h}_{\alpha} \coloneqq \ker(\alpha) \subset \mathbb{h}$ and $\tT_\alpha \coloneqq \ker(e^\alpha) \subset\tT$.

For $f = \sum_{w\in \tW} f_w [w]\in Q[\tW]$ consider the following conditions.

(1) For any $w\in \tW$, the rational function $f_w$ has no singularities other than finitely many first order poles at $\th_{\alpha}$, where $\alpha \in \bbPhire_+$

(2a) For any $w\in \tW$ and $\alpha \in \bbPhire_+$, we have $\Res_{\th_{\alpha}}(f_w) + \Res_{\th_\alpha} (f_{s_{\alpha}w}) = 0$.

(2b) For any $w\in \tW$ and $\alpha \in \bbPhire_+$, we have $\Res_{\th_{\alpha}}\left(\frac{\Pi_\Sigma}{\Pi_{w(\Sigma)}}f_w\right) + \Res_{\th_\alpha} \left(\frac{\Pi_\Sigma}{\Pi_{s_\alpha w(\Sigma)}}f_{s_{\alpha}w}\right) = 0$.

(3) For any $w\in \tW$, the rational function $\frac{f_w}{\tPi_{w(\Sigma),\Sigma}}$ has no singularities other than finitely many first order poles at $\th_{\alpha}$, where $\alpha \in \bbPhire_+$

(1) and (2a) are conditions (1.3.1) and (1.3.2) in \cite{GKV}, and (3) is the generalization to $\Sigma$ of condition (1.3.3). In $K$-theory, we consider similar conditions, replacing $\th_\alpha$ by $\tT_\alpha$.

\noindent The following follows from \cite{GKV}, see also \cite[3(vi)]{FT}.

\begin{proposition}
\label{tH_(1)(2)}
    $\H = \{ f\in Q[\tW] \mid f\text{ satisfies conditions (1), (2a)} \}$
\end{proposition}

\begin{proof}
    The right hand side is an algebra by Theorem 1.4 in \cite{GKV}. Every element of $S$ and each $D_{s_i}$ satisfies conditions (1) and (2a), so $\H \subset \{ f\in Q[\tW] \mid f\text{ satisfies (1) and (2a)} \}$. Using Proposition \ref{tH_description}, the reverse inclusion follows from Lemma 3.1(i) in \cite{GKV}.
\end{proof}

\begin{theorem}
\label{H_Sigma_(1)(2)(3)}
    $\H_\Sigma = \{ f \in Q[\tW] \mid f \text{ satisfies conditions (1), (2a), (2b), (3)}\}$
\end{theorem}

\begin{proof}
    Let $f = \sum_{w\in \tW} f_w [w] \in Q[\tW]$ such that it satisfies (1), (2a), (2b), (3). By Proposition \ref{tH_(1)(2)}, $f\in \H$. Now consider
    $$\Pi_\Sigma f \Pi_\Sigma^{-1} = \sum_{w\in \tW} f_w \frac{\Pi_\Sigma}{\Pi_{w(\Sigma)}} [w] \in Q[\tW]$$

    \noindent It satisfies condition (2a), because $f$ satisfies condition (2b). Since $f_w\frac{\Pi_\Sigma}{\Pi_{w(\Sigma)}} = \frac{f_w}{\tPi_{w(\Sigma),\Sigma}} \tPi_{\Sigma,w(\Sigma)}$, condition (3) for $f$ implies condition (1) for $\Pi_\Sigma f \Pi_\Sigma^{-1}$. Using Proposition \ref{tH_(1)(2)}, we obtain $\Pi_\Sigma f \Pi_\Sigma^{-1} \in \H$. So $f \in \H_\Sigma$.

    Conversely, let $f = \sum_{w\in \tW} f_w [w] \in \H_\Sigma$. In particular, $f\in \H$, so Proposition \ref{tH_(1)(2)} implies that $f$ satisfies (1) and (2a). As $\Pi_\Sigma f \Pi_\Sigma^{-1} \in \H$, it satisfies condition (2a), so $f$ satisfies condition (2b). We check that it also satisfies condition (3). Since $f\in\H_\Sigma$, we have:
    $$\Pi_\Sigma f \Pi_\Sigma^{-1} = \sum_{w\in \tW} f_w \frac{\Pi_\Sigma}{\Pi_{w(\Sigma)}}[w] =  \sum_{w\in \tW} f_w \frac{\tPi_{\Sigma, w(\Sigma)}}{\tPi_{w(\Sigma), \Sigma}}[w] \in \H$$
    
    Since $\tPi_{\Sigma, w(\Sigma)}$ and $\tPi_{w(\Sigma), \Sigma}$ are coprime, condition (1) for $\Pi_\Sigma f \Pi_\Sigma^{-1}$ and $f$ implies condition (3) for $f$.
\end{proof}

\begin{remark}
    Assume that, for all $\alpha \in \bbPhire_+$, no element of $\C \alpha$ is in $\Sigma$. Then, assuming condition (1), conditions (2a) and (2b) become equivalent, so Theorem \ref{H_Sigma_(1)(2)(3)} says:
    $$\H_\Sigma = \{ f \in Q[\tW] \mid f \text{ satisfies conditions (1), (2a), (3)}\}$$
    This is the case considered in \cite{GKV} and this is the case of interest in our future geometric applications, see Remarks \ref{finite_gluability_remark} and \ref{finite_Springer_remark}.
\end{remark}

\begin{lemma}
\label{correct_leading_term}
    Let $f = g_w D_w + \sum_{c \not\geq w} g_c D_c \in \H_\Sigma$. Then its leading term $g_w$ is divisible by $\tPi_{w(\Sigma),\Sigma}$ in $S$.
\end{lemma}

\begin{proof}
    Let $\alpha = \prod_{\beta \in \mathbb{\Phi}_+ \cap w \mathbb{\Phi}_-} \beta$. By Theorem 11.1.2 (e2) in \cite{Ku}, the leading term of $D_w$ is $\frac{(-1)^{\ell(w)}}{\alpha} [w]$, i.e. $D_w = \frac{(-1)^{\ell(w)}}{\alpha} [w] + \sum_{c< w} d_{w,c} [c]$ for some $d_{w,c} \in Q$. So $f = \frac{(-1)^{\ell(w)}g_w}{\alpha} [w] + \sum_{c \not\geq w} f_c [c]$ for some $f_c\in Q$, and
    
    $$\Pi_\Sigma f \Pi_{\Sigma}^{-1} = \frac{(-1)^{\ell(w)}g_w}{\alpha} \frac{\tPi_{\Sigma,w(\Sigma)}}{\tPi_{w(\Sigma),\Sigma}}[w] + \sum_{c\not 
    \geq w} f_c \frac{\Pi_\Sigma}{\Pi_{c(\Sigma)}}[c] \in \H$$
    
    \noindent Since $\H$ is a free $S$-module with basis $\{D_w\}_{w\in \tW}$, we obtain
    $$\frac{g_w}{\alpha} \frac{\tPi_{\Sigma,w(\Sigma)}}{\tPi_{w(\Sigma),\Sigma}} \in \frac{1}{\alpha} S$$
    Since $\tPi_{\Sigma,w(\Sigma)}$ and $\tPi_{w(\Sigma),\Sigma}$ are coprime, this implies that $\tPi_{w(\Sigma),\Sigma} | g_w$ in $S$.
\end{proof} 

\begin{proposition}
    \label{H_Sigma_basis}
    Assume that for each $w\in \tW$, there exists some $A_w \in \H_\Sigma$ such that \mbox{$A_w= \tPi_{w(\Sigma),\Sigma} D_w + \sum_{c < w} g_{w,c} D_c$}, where $g_{w,c} \in S$. Then $\{ A_w\}_{w\in\tW}$ is a basis of $\H_\Sigma$ viewed as an $S$-module.
\end{proposition}

\begin{proof}
The elements $D_w$ are linearly independent over $Q$, so $A_w$ are also linearly independent. We prove that they span $\H_\Sigma$ by induction with respect to the Bruhat order. Let $f = \sum_{c\in \tW} g_c D_c \in \H_\Sigma$, with $g_c \in S$. Pick some $w\in \tW$ such that $f = g_w D_w + \sum_{c\not \geq w} g_c D_c$. Lemma \ref{correct_leading_term} implies that $\frac{g_w}{\tPi_{w(\Sigma),\Sigma}}  \in S$. The coefficient of $[w]$ in $f - \frac{g_w}{\tPi_{w(\Sigma),\Sigma}}  A_w$ is $0$, so the conclusion follows by induction.
\end{proof}

\begin{remark}
    It is not clear to us whether such a basis always exists. The only case in which we know it does is when such a basis comes from geometry, i.e. when $\Sigma$ is the collection of weights of a representation $M$, for a gluable pair $(V,M)$. \\
    The algebra $\H_\Sigma$ is a $\tW$-filtered algebra, as a subalgebra of the $\tW$-graded algebra $Q[\tW]$. The existence of a basis as in Proposition \ref{H_Sigma_basis} is equivalent to an isomorphism of graded $S$-modules:
    $$\operatorname{gr}(\H_\Sigma) \cong \bigoplus_{w\in \tW} \frac{\tPi_{w(\Sigma),\Sigma}}{\prod_{\beta \in \mathbb{\Phi}_+ \cap w \mathbb{\Phi}_-} \beta} \cdot S \cdot [w]$$
\end{remark}

\subsection*{An anti-involution and a bilinear form in the finite type}

Consider the anti-involution $\tau_0 : \H \rightarrow \H$ defined by 
$$\tau_0\left(\sum_{w\in \tW} g_w D_w\right) = \sum_{w\in \tW} D_{w^{-1}} g_w$$
Let
$$\Delta \coloneqq \prod_{\alpha \in \mathbb{\Phi}_+} \alpha$$
Then $\tau_0$ is the restriction to $\H$ of the anti-involution $\tau_0 : Q[\tW]\rightarrow Q[\tW]$ defined by 
$$\tau_0(f[w]) = \Delta [w^{-1}] \Delta^{-1} f= (-1)^{\ell(w)} [w^{-1}] f$$

\noindent for $f\in Q$. Define the anti-involution 
$\tau_\Sigma : Q[\tW] \rightarrow Q[\tW]$ by $$\tau_\Sigma = \Pi_\Sigma^{-1} \tau_0 \Pi_\Sigma$$

Since $\tau_\Sigma(\H) = \Pi_\Sigma^{-1} \H \Pi_\Sigma$, we have $\H_\Sigma = \H \cap \tau_{\Sigma}(\H)$. Therefore, $\tau_\Sigma$ restricts to an anti-involution of $\H_\Sigma$.

\noindent The $K$-theoretic analogue of $\Delta$ is
$$\varDelta \coloneqq \prod_{{\alpha \in \mathbb{\Phi}_+}} (1-e^{-\alpha})$$

Now assume that the initial generalized Cartan matrix is of finite type, so $\mathbb{G}$ is a finite dimensional reductive group and $\tW=W$ is its Weyl group. Also assume that $\Sigma$ is finite. Define the symmetric bilinear form $\langle - , -\rangle_\Sigma : Q \times Q \rightarrow Q^W$ by
$$\langle f,g\rangle_\Sigma \coloneqq \sum_{w\in W} w\left( \frac{fg\Pi_\Sigma}{\Delta}\right) = \frac{\sum_{w\in W}(-1)^{\ell(w)} \cdot w(fg\Pi_\Sigma)}{\Delta}$$

\noindent Recall that $Q[W]$ acts on $Q$. A direct computation shows that $\tau_\Sigma : Q[W] \rightarrow Q[W]$ satisfies

$$\langle z.f, g \rangle_\Sigma = \langle f, \tau_\Sigma(z).g\rangle_{\Sigma}$$

\noindent for all $f,g\in Q$ and $z\in Q[W]$.

\noindent Moreover, the same formula defines a bilinear map $\langle - , -\rangle_\Sigma : S \times S\Pi_{\Sigma}^{-1} \rightarrow S^W$. Recall that $\H_\Sigma$ acts on $S$ and on $S \Pi_\Sigma^{-1}$. Then $\tau_\Sigma : \H_\Sigma \rightarrow \H_\Sigma$ satisfies a similar adjointness equation.

\subsection*{A spherical version of $\H_\Sigma$}

Let $J\subset \mathbb{\Pi}$ be a subset of simple roots generating a root subsystem of finite type, so $W_J \subset W$ is finite. Let $e = \frac{1}{|W_J|} \sum_{w\in W_J} [w] \in \C[W_J] \subset \H$. As $e$ is idempotent in $\H$, $e\H e$ is an algebra with unit $e$. Recall the action of $Q[\tW]$ on $Q$, and consider $eQ \subset Q$. The restriction of the action of $Q[\tW]$ makes $eQ$ into an $e Q[\tW] e$-module. We also consider $eS \subset eQ$. Explicitly $eS\cong S^{W_J}$ and $eQ \cong Q^{W_J}$, but we will not use this.

\begin{lemma}
\label{tHsph_description}
    $e \H e = \{ \tf \in e Q[\tW] e \mid \tf.(eS) \subset eS\}$
\end{lemma}

\begin{proof}
    If $\tf = efe$, with $f\in \H$, then
    $$\tf .eS = efe.S \subset ef.S \subset e.S$$
    because $f.S\subset S$.
    
    Conversely, let $\tf \in e Q[\tW]e$ such that $\tf.(eS) \subset eS$.
    Complete $e=e_1$ to a set of orthogonal idempotents $e_1, e_2, \ldots e_k \in \C[W_J] \subset \H$ with $e_1 + e_2 + \ldots + e_k = 1 $. So $S = e_1S \oplus e_2S \oplus \ldots \oplus e_k S$. Since $\tf. e_i S =0$ for $i \geq 2$, we have $\tf . S \subset S$. Then, by Proposition \ref{tH_description}, we have $\tf \in \H$, so $\tf = e \tf e \in e\H e$.
\end{proof}

We assume that $\Sigma$ is $W_J$-invariant, so $\Pi_{\Sigma}^{-1} e \Pi_\Sigma = e$. Thus $\Pi_\Sigma^{-1} (-) \Pi_\Sigma$ restricts to an automorphism of $e Q[\tW] e$. Recall the $Q[\tW]$-module $Q \Pi_\Sigma^{-1}$ defined before. We consider its subset $eS \Pi_\Sigma^{-1} \coloneqq \{ e g \Pi_\Sigma^{-1} \mid g \in S\}$.

\begin{lemma}
\label{H_Sigma_sph_description}
    $$e \H_\Sigma e = e\H e \cap \Pi_\Sigma^{-1} e \H e \Pi_\Sigma = \{\tf\in e Q[\tW] e \mid \tf . (eS) \subset eS\text{ and } \tf.(eS\Pi_\Sigma^{-1}) \subset eS\Pi_\Sigma^{-1}\} $$
\end{lemma}

\begin{proof}
The second equality follows from Lemma \ref{tHsph_description}, because $\tf.(eS\Pi_\Sigma^{-1}) \subset eS\Pi_\Sigma^{-1}$ if and only if $(\Pi_\Sigma \tf \Pi_\Sigma^{-1}) . (eS) \subset eS$. Now, we prove the first equality.

Multiplying  $\H_\Sigma = \H \cap \Pi_\Sigma^{-1} \H \Pi_\Sigma$ by $e$ on both sides, we have $e \H_\Sigma e \subset e \H e \cap \Pi_\Sigma^{-1} e \H e \Pi_\Sigma$.

Conversely, let $\tf\in e Q[\tW] e$ such that $\tf . (eS) \subset eS\text{ and } \tf.(eS\Pi_\Sigma^{-1}) \subset eS\Pi_\Sigma^{-1}$. As in the proof of Lemma \ref{tHsph_description}, the first inclusion implies $\tf \in \H$ and the second inclusion implies $\Pi_\Sigma \tf \Pi_\Sigma^{-1} \in \H$. That is, $\tf \in \H_\Sigma \coloneqq \H \cap \Pi_\Sigma^{-1} \H \Pi_\Sigma$. So $\tf = e \tf e \in e \H_\Sigma e$.
\end{proof}

\section{Geometric construction of $\H_\Sigma$}
\label{section_finite}

Let $G$ be a complex reductive group. Fix $T\subset B\subset G$ a maximal torus and a Borel subgroup, and let $\B\cong G/B$ be the flag variety of $G$. Let $\Phi, \Phi_+, \Pi$ be the set of roots, positive roots and simple roots, respectively. Let $V$ be a finite dimensional representation of $G$ and let $M \subset V$ be a $B$-stable subspace. Consider the vector bundle $\tM \coloneqq G\times^B M \xrightarrow{\pi} G/B\cong \B$. The map $\tM \rightarrow \B \times V$, $(gB,v) \mapsto (gB,g.v)$ is a closed embedding of $\tM$ into the trivial vector bundle $\B \times V$. Composing with the projection to the second factor, we obtain a proper map $\mu:\tM \rightarrow V$, $(gB,v) \mapsto g.v$.

We define the following analogue of the Steinberg variety $Z\coloneqq \tM \times_V \tM$. The map $Z \rightarrow \B \times\B\times V$, $(g_1B,g_2B,v) \mapsto (g_1B,g_2B, g_1.v)$ is a closed embedding. Projecting on the first two factors, we obtain a map $\pi^2 : Z\rightarrow \B\times \B$. The Bruhat decomposition gives $\B \times \B = \bigsqcup_{w\in W} (\B\times\B)_w$, where $(\B\times\B)_w\coloneqq G.(1B,wB)$. So $Z = \bigsqcup_{w\in W} Z_w$, where \mbox{$Z_w \coloneqq (\pi^2)^{-1}(\B\times \B)_w$}. A direct computation shows that $(\pi^2)^{-1}(1B, wB) = M \cap wM \subset V$, where $wM = \dot{w}M$ for some lift $\dot{w}\in G$ of $w\in W$. Since $\pi^2$ is $G$-equivariant, we obtain an isomorphism of $G$-equivariant vector bundles:

\begin{equation}
\label{finite_Zw_as_bundle}
\begin{tikzcd}
	Z_w & G \times^{B \cap wB w^{-1}} (M \cap wM)\\
    \\
    (\B\times\B)_w & G/(B\cap w B w^{-1})
	\arrow["\cong"', from=1-1, to=1-2]
	\arrow["\pi^2", from=1-1, to=3-1]
	\arrow[""', from=1-2, to=3-2]
	\arrow["\cong", from=3-1, to=3-2]
\end{tikzcd}  
\end{equation}

\begin{remark}
    For the adjoint representation $V=\g$ and $M = \mathfrak{n} \coloneqq \operatorname{Lie} [B,B]$, we recover the Springer resolution $\tM \cong T^* \B$ and $Z$ is the usual Steinberg variety. Then $\tM \times \tM \cong T^*(\B\times \B)$ and $Z_w$ is the conormal bundle to $(\B \times \B)_w$, see \textsection 3.3 in \cite{CG}.\\
    The case $V=\g$, $M=\mathfrak{b}\coloneqq \operatorname{Lie} B$ corresponds to the Grothendieck-Springer simultaneous resolution.
\end{remark}

For a variety $X$ with an action of an algebraic group $L$, denote by $\HBM{L}{X}$ its $L$-equivariant Borel-Moore homology with coefficients in $\C$. This is a module over $H^\bullet_L(\pt)$.

We use convolution in equivariant Borel-Moore homology, as in \textsection 2.7 in \cite{CG}.  We have $Z\subset \tM \times \tM$ and $Z\circ Z = Z$ inside $\tM \times \tM \times \tM$, so convolution in Borel-Moore homology endows $\HBM{G}{Z}$ with an $H_G^\bullet(\pt)$-algebra structure. Also, $Z \circ \B = \B$ inside $\tM \times \tM \times \{\pt\}$, so $\HBM{G}{\B}$ is a module over $\HBM{G}{Z}$. Here we identify $\B$ with the $0$-section in $\tM\times\{ \pt\}$.

We use notation from Section \ref{section_algebra} for $\mathbb{G}=G$. Explicitly, let $S = \C[\t]$, $Q=\C(\t)$, and let $\mathbb{W}=W$ be the Weyl group of $G$. For a $T$-representation $N$, let $\Sigma(N)$ be the collection of weights of $N$, counted with multiplicity. For brevity, put  $\Pi_{N} \coloneqq \Pi_{\Sigma(N)}$ and $\H_N \coloneqq \H_{\Sigma(N)}$. For the pair $(V,M)$ and for $w\in W$, the subspace $wM\coloneqq \dot{w} M$ is a $T$-subrepresentation of $V$, where $\dot{w} \in G$ is any lift of $w\in W$. We have $\Sigma(wM) = w(\Sigma(M))$. Let $\Pi_{wM}\coloneqq  \Pi_{\Sigma(wM)}= \Pi_{w(\Sigma(M))}$.

\begin{definition}
\label{finite_gluability_def}
    A pair $(V,M)$ is called gluable if $\frac{\Pi_{M}}{\Pi_{M\cap wM}}$ and $\frac{\Pi_{wM}}{\Pi_{M\cap wM}}$ are coprime in $S$, for all $w\in W$.
\end{definition}

The property of a pair $(V,M)$ being gluable depends not only on $M$, but also on the ambient space $V$, since the intersection $M\cap wM$ is taken inside $V$. If the pair $(V,M)$ is gluable, then one has equalities $\tPi_{M,wM} = \frac{\Pi_M}{\Pi_{M\cap wM}}$ and $\tPi_{wM,M} = \frac{\Pi_{wM}}{\Pi_{M\cap wM}}$. As explained below, $\HBM{G}{\B}$ is a free $S$-module of rank $1$ generated by the fundamental class $[\B]$. The goal of this section is to prove the following.

\begin{theorem}
\label{finite_H(Z)=H_M}
    Assume that the pair $(V,M)$ is gluable. Then there exists an algebra isomorphism $\HBM{G}{Z} \cong\H_M$. This isomorphism intertwines the action of $\HBM{G}{Z}$ on $\HBM{G}{\B}$ by convolution and the action of $\H_M$ on $S$.
\end{theorem}

If $G$ has a simply connected derived group, the analogous theorem holds when equivariant Borel-Moore homology is replaced with equivariant $K$-theory with coefficients in $\C$. The proof is similar, using results in \cite{KL} and \cite{CG}.

\begin{remark}
\label{finite_gluability_remark}
Let $G = G' \times \Cdil$, where $G'$ is a reductive group and $\Cdil=\C^*$. Let $(V,M)$ be a pair such that $\Cdil$ acts on $V$ by dilations $z.v\coloneqq zv$. Let $T=T'\times \Cdil$, where $T'$ is a maximal torus in $G'$. In this case, gluability of $(V,M)$ is equivalent to the property that for all $w\in W$ and $\lambda \in X^*(T')$, $\dim (M\cap wM)_\lambda = \min(\dim M_\lambda, \dim (wM)_\lambda)$, where $(-)_\lambda$ is the $\lambda$-weight space. This holds if, for each $\lambda\in X^*(T')$, the space $M_\lambda$ is either $V_\lambda$ or $0$.
\end{remark}

\begin{remark}
\label{finite_Springer_remark}
    In the case of the Springer resolution $(V,M)=(\g',\mathfrak{n}')$, Theorem \ref{finite_H(Z)=H_M} and its $K$-theoretic analogue recover the geometric realizations of the graded/degenerate affine Hecke algebra \cite{Lu} and the affine Hecke algebra \cite{KL}, \cite{Gi2}. Constructions of these algebras in terms of poles and residues are given in \cite{GKV}. The algebra $\H_q$ studied in \cite{GKV} and our $\H_M$ are related by $\H_q = \Pi_M \H_M \Pi_M^{-1}$ inside $Q[W]$. Conditions (1.3.1) and (1.3.2) in \cite{GKV} correspond to our conditions (1) and (2a), and condition (1.3.3) corresponds to:
    
    \noindent (3'): For any $w\in \tW$, the rational function $\frac{f_w}{\tPi_{\Sigma,w(\Sigma)}}$ has no singularities other than finitely many first order poles at $\th_{\alpha}$, where $\alpha \in \bbPhire_+$.
    
    \noindent A straightforward modification of the argument in Theorem \ref{H_Sigma_(1)(2)(3)} shows that 
    $$\Pi_M\H_M \Pi_M^{-1} = \{ f \in Q[W] \mid f \text{ satisfies conditions (1), (2a), (3')}\}.$$
\end{remark}

The swap morphism $\sw:Z \rightarrow Z$ defined by $(a,b)\mapsto (b,a)$ induces an anti-involution $\sw_*:\HBM{G}{Z} \rightarrow \HBM{G}{Z}$. Let $
    \langle -,- \rangle : \HBM{G}{\B} \times \HBM{G}{\B} \rightarrow \HBM{G}{\pt}$ denote convolution in $\{ \pt \} \times \tM \times \{ \pt \}$.

\begin{proposition}
    The map $\sw_* : \HBM{G}{Z} \rightarrow \HBM{G}{Z}$ is the adjoint anti-involution with respect to $\langle -,- \rangle$. Under the equivariant localization theorem, $\sw_*$ and $\langle -, -\rangle$ are identified with $\tau_{\Sigma(M)}$ and $\langle-,-\rangle_{\Sigma(M)}$ defined in Section \ref{section_algebra}.
\end{proposition}

\begin{proof}
    Using the definition of convolution, for all $z\in \HBM{G}{Z}$ and $a\in \HBM{G}{\B}$, we have $z \star a =a \star \sw_*(z)$, where the first convolution is in $\tM\times\tM\times\{\pt\}$ and the second convolution is in $\{\pt\}\times\tM\times\tM$. Then, the associativity of convolution in $\{ \pt \} \times \tM \times \tM \times \{ \pt\}$ implies that $\sw_*$ is the adjoint anti-involution with respect to $\langle -, - \rangle$.

    The identification of $\langle-, -\rangle$ with $\langle-,-\rangle_{\Sigma(M)}$ follows from direct computations, as $\B$ is smooth. The identification of $\sw_*$ and $\tau_{\Sigma(M)}$ follows from Lemma \ref{finite_leading_term_ZW}, to be proved later, observing that $\sw_*\left([\overline{Z_w}]\right) = [\overline{Z_{w^{-1}}}]$.
\end{proof}

\subsection*{Proof of Theorem \ref{finite_H(Z)=H_M}}

We identify $H^\bullet_T(\pt) \cong \C[\t] = S$. Let $Z'\coloneqq \tM \cap (G/B\times M) \subset G/B \times V$, where $\tM \hookrightarrow G/B \times V$ is as before. A direct computation shows that the morphism $\alpha : G\times^B Z' \rightarrow G/B \times G/B \times V$, $(gB,z) \mapsto (gB, g.z)$ induces an isomorphism $G\times^B Z' \cong Z$, which commutes with projections on the first factor $G/B$. Let $Z_\Delta = Z_1 \coloneqq (\pi^2)^{-1}(\B_\Delta)$, where $\B_\Delta \subset \B \times \B$ is the diagonal copy of $\B$. Then, $Z_\Delta \circ Z_\Delta = Z_\Delta$, so $\HBM{G}{Z_\Delta}$ acquires an algebra structure. Since $Z_\Delta \cong \tM$ is smooth, $\HBM{G}{\tM}$ is the free $H^\bullet_{G}(\tM)$-module of rank $1$ generated by the fundamental class $[\tM]$. This interchanges the cup product in $H^\bullet_{G}(\tM)$ and the convolution algebra structure of $\HBM{G}{\tM}$, see Example 2.7.10(i) in \cite{CG}. We have $Z_\Delta \circ Z = Z$ inside $\tM \times \tM \times \tM$, so $\HBM{G}{Z}$ is a module over $\HBM{G}{Z_\Delta} \cong \HBM{G}{\tM}$. 

For a variety $X$ equipped with an action of $B$, there exists an induction isomorphism {$\operatorname{ind}_X :\HBM{B}{X} \xrightarrow{\cong} \HBM{G}{G\times^B X}$}, see \cite[\textsection 2.6]{BL} and \cite[\textsection 5.2.16]{CG}. As $B$ homotopy retracts to $T$, we have $\HBM{B}{X} \cong \HBM{T}{X}$. For $X=Z'$, we obtain an isomorphism $\HBM{G}{Z} \cong \HBM{B}{Z'} \cong \HBM{T}{Z'}$. For $X=M$, we obtain isomorphisms $$\HBM{G}{Z_\Delta} \cong \HBM{G}{\tM} \cong \HBM{B}{M} \cong \HBM{T}{M} = H^\bullet_T(\pt).[M] = S.[M]$$ where $H^\bullet_T(\pt).[M]$ is the free module of rank $1$ generated by the fundamental class of $M$.

\begin{lemma}
\label{finite_S_actions_agree}
    The action of $\HBM{G}{Z_\Delta}$ on $\HBM{G}{Z}$ by convolution is identified with the action of $H^\bullet_T(\pt)$ on $\HBM{T}{Z'}$ via the isomorphisms $\HBM{G}{Z} \cong \HBM{T}{Z'}$ and $\HBM{G}{Z_\Delta} \cong \text{H}^\bullet_T(\pt).[M]$ described above.
\end{lemma}

\begin{proof}
    Let $a = \alpha \cap[Z_\Delta] \in \HBM{G}{Z_\Delta}$, where $\alpha\in \text{H}^\bullet_G(Z_\Delta)$. Let $b\in \HBM{G}{Z}$. 
    Let $p_1:Z\rightarrow \tM$ be the restriction to $Z$ of the first projection $\tM\times\tM \rightarrow \tM$. By base change and the projection formula, we have $a\star b = p_1^*(\alpha) \cap b$.

    Let $a=\text{ind}_M(\alpha'\cap [M])$ and $b=\text{ind}_{Z'}(b')$, where $\alpha'\in \text{H}^\bullet_B(M)$ and $b'\in \HBM{B}{Z'}$. Since $\text{ind}$ commutes with $\cap$ and with flat pullbacks, we have $a\star b = \text{ind}_{Z'}(q_1^* (\alpha') \cap b')$, where $q_1 : Z' \rightarrow M$ is the restriction to $Z'$ of the projection $\B\times M \rightarrow M$. Let $q:M\rightarrow \{\pt\}$ be the projection to a point. As $M$ is a vector space, it is contractible, so $q^* : \text{H}^\bullet_{B}({\pt}) \rightarrow \text{H}^\bullet_B(M)$ is an isomorphism. Thus $q_1^*(\alpha') \cap b' = (q^*)^{-1}(\alpha') . b'$, which is what we wanted to prove.
\end{proof}

Thanks to this lemma, we view $\HBM{G}{Z}$ as an $S$-module. The following is the homological version of Theorem 6.2.4 in \cite{CG}.

\begin{lemma}
\label{finite_S_mod_basis}
    The fundamental classes $\left\{\left[\overline{Z_w}\right]\right\}_{w\in W}$ form a basis of the $S$-module $\HBM{G}{Z}$.
\end{lemma}

\begin{proof}
    Let $\pi:Z' \rightarrow G/B$ be the restriction of $G/B \times V \rightarrow G/B$. For $w\in W$, denote $Z'_w \coloneqq \pi^{-1}(BwB/B)$. As before, $\pi$ restricts to a $B$-equivariant vector bundle $Z'_w \rightarrow BwB/B$. In particular, $Z'_w$ is an affine space. Moreover, the isomorphism $G\times^BZ' \cong Z$ restricts to an isomorphism $G\times^B Z'_w \cong Z_w$. Thus, using Lemma \ref{finite_S_actions_agree}, it is enough to prove that the classes $\left\{\left[\overline{Z'_w}\right]\right\}_{w\in W}$ form a basis of the $S$-module $\HBM{T}{Z'}$.

    Completing the Bruhat order to a total linear order, we obtain a paving of $Z'$ by $T$-stable affine spaces $Z'_w$. Since the Borel-Moore homology of the strata is concentrated in even degrees, the associated long exact sequences break into short exact sequences. Thus, $\HBM{T}{Z'}$ is a free $S$-module with a basis given by $\left\{\left[\overline{Z'_w}\right]\right\}_{w\in W}$.
\end{proof}

Let $\rho_0 : \HBM{G}{Z} \rightarrow \End_{H_G^\bullet(\pt)} \left( \HBM{G}{\B}\right)$ be the convolution map induced by $Z\circ \B = \B$ in $\tM \times \tM \times \{ \pt \}$. Let $\rho : \HBM{G}{\B\times\B}\rightarrow \End_{H_G^\bullet(\pt)} \left( \HBM{G}{\B}\right)$ be the convolution map induced by $(\B\times \B)\circ \B = \B$ in $\B \times \B \times \{ \pt \}$. Let $p:Z\rightarrow \B\times\tM$ be the composition of $Z\hookrightarrow \tM\times\tM$ and $\pi\times\text{id}_{\tM} : \tM\times\tM \rightarrow\B\times\tM$. As $p$ is a closed embedding, it is proper. Let $i : \B\times\B \rightarrow \B\times\tM$ be induced by the $0$-section embedding $\B\hookrightarrow \tM$ on the second factor. Let $\iota_0$ be the composition of the maps $\HBM{G}{Z} \xrightarrow{p_*} \HBM{G}{\B\times\tM} \xrightarrow{i^*} \HBM{G}{\B\times\B}$, where $i^*$ is the Gysin pullback. This setup is similar to \cite[\textsection 5.4.22]{CG}, but in our case the intermediary space is $\B\times\tM$, and in their case it is $\tM \times \B$. Adapting the proof of Lemma 5.4.27 in \cite{CG}, we obtain:

\begin{lemma}
\label{finite_map_to_H0}
    The following diagram commutes
    \[\begin{tikzcd}
	\HBM{G}{Z} \\
	&&  \End_{H_G^\bullet(\pt)} \left( \HBM{G}{\B}\right) \\
	\HBM{G}{\B\times\B}
	\arrow["\rho_0", from=1-1, to=2-3]
	\arrow["\iota_0"', from=1-1, to=3-1]
	\arrow["\rho"', from=3-1, to=2-3]
\end{tikzcd}\]
\end{lemma}

As before, using the induction isomorphism in equivariant homology, we identify $$\HBM{G}{\B} \cong \HBM{G}{G/B} \cong \HBM{B}{\pt} \cong \HBM{T}{\pt} = S$$ $$H_G^\bullet(\pt) \cong \left( H_T^\bullet(\pt) \right)^W \cong S^W$$ Let $\alpha:\End_{H_G^\bullet(\pt)} \left( \HBM{G}{\B}\right) \xrightarrow{\cong} \End_{S^W}(S)$ be the resulting isomorphism. 

The next proposition follows from \cite{KK1}.

\begin{proposition}
\label{finite_H_0_thm}
The assignment
$$\left[ \overline{(\B\times\B)_w} \right] \mapsto D_w$$
extends to an $S$-linear algebra isomorphism $\theta : \HBM{G}{\B\times \B} \xrightarrow{\cong}\H$, such that the following diagram commutes:
    \[\begin{tikzcd}
	\HBM{G}{\B\times\B} && \End_{H_G^\bullet(\pt)} \left( \HBM{G}{\B}\right) \\
	\\
	\H && \End_{S^W}(S)
	\arrow["\rho"', from=1-1, to=1-3]
	\arrow["\theta", from=1-1, to=3-1]
	\arrow["\alpha", from=1-3, to=3-3]
	\arrow["\rho'",from=3-1, to=3-3]
\end{tikzcd}\]
    Here $\rho'$ is the action of $\H$ on $S$ described in Section \ref{section_algebra}. Moreover, all maps in the above diagram are isomorphisms.
\end{proposition}

Let $\theta_0 \coloneqq \theta \circ\iota_0: \HBM{G}{Z}\rightarrow \H$.

\begin{lemma}
\label{finite_Image_in_HM}
    $\Im(\theta_0) \subset \H_M$.
\end{lemma}

\begin{proof}
    Combining Lemma \ref{finite_map_to_H0} and Proposition \ref{finite_H_0_thm}, we have $\rho' \circ \theta_0 = \alpha \circ \rho_0$. Using Proposition \ref{H_Sigma_is_normalizer}, it is enough to prove that for all $z\in \HBM{G}{Z}$, the map $(\alpha\circ \rho_0)(z)$ sends the fractional ideal $S \Pi_M^{-1}$ to itself.

    Recall the map $\rho_0$, given by convolution, using $Z\circ\B = \B$ in $\tM \times \tM \times \{\pt\}$. We also have $Z\circ \tM = \tM$, so we obtain $\rho_M : \HBM{G}{Z} \rightarrow\End_{H_G^\bullet(\pt)}\left(\HBM{G}{\tM}\right)$. Let $\varepsilon_* : \HBM{G}{\B} \rightarrow \HBM{G}{\tM}$ be the pushforward along the zero section embedding. By the Thom isomorphism, $\varepsilon_*$ is injective and its image is the principal ideal in $\HBM{G}{\tM} = S.[\tM]$ generated by $\Pi_M$. Here $\Pi_M \in S$ is the equivariant Euler class of $\tM \rightarrow \B$. Lemma 5.2.23 in \cite{CG} shows that the following diagram commutes:
    \[\begin{tikzcd}
	\HBM{G}{Z} \otimes \HBM{G}{\B} && \HBM{G}{\B} \\
	\\
	\HBM{G}{Z} \otimes \HBM{G}{\tM} && \HBM{G}{\tM}&& 
	\arrow["\star"', from=1-1, to=1-3]
	\arrow["\text{id} \otimes \varepsilon_*", from=1-1, to=3-1]
	\arrow["\varepsilon_*", from=1-3, to=3-3]
	\arrow["\star",from=3-1, to=3-3]
\end{tikzcd}\]

\noindent where the top horizontal map is induced by $\rho_0$ and the bottom one by $\rho_M$. This implies that for all $z\in \HBM{G}{Z}$, $(\alpha\circ \rho_0)(z)$ maps the fractional ideal $S \Pi_M^{-1}$ to itself.

\end{proof}

\begin{lemma}
\label{finite_leading_term_ZW}
    The action of $\theta_0$ on fundamental classes is given by:
    $$\theta_0\left(\left[\overline{Z_w}\right]\right) = \frac{\Pi_{wM}}{\Pi_{M\cap wM}} D_w + \sum_{y < w} b_{w,y} D_y$$
    for some $b_{w,y} \in S$.
    In particular, $\theta_0$ is injective.
\end{lemma}

\begin{proof}
    Fix $w\in W$. Let $Y_w \coloneqq (\B\times\B)_w$. Consider the open-closed embeddings $$Y_w \xhookrightarrow{j} \overline{Y_w} \xhookleftarrow{k} \overline{Y_w} \setminus Y_w$$ As in the proof of Lemma \ref{finite_S_mod_basis}, the corresponding long exact sequence in homology breaks into short exact sequences:
    \begin{center}
    \begin{tikzcd}[column sep=large]
    0 \arrow[r] & \HBM{G}{\overline{Y_w} \setminus Y_w} \arrow[r, "k_*"] & \HBM{G}{\overline{Y_w}} \arrow[r, "j^*"] & \HBM{G}{Y_w} \arrow[r] & 0
    \end{tikzcd}    
    \end{center}
    Using Proposition \ref{finite_H_0_thm}, it is enough to prove that 
    $$j^* i^* p_*([\overline{Z_w}]) = \frac{\Pi_{wM}}{\Pi_{M\cap wM}} [Y_w]$$ 
    Let $\pi' : G\times^{B\cap w B w^{-1}} wM \rightarrow Y_w$ be the restriction of the morphism $\B \times \tM \rightarrow \B\times\B$ to $Y_w$. Let $i' : Y_w \rightarrow G\times^{B\cap w B w^{-1}} wM$ be the $0$-section of the vector bundle $\pi'$. Denote by $p' : Z_w \hookrightarrow G\times^{B\cap w B w^{-1}} (wM)$ the restriction of $p : Z \hookrightarrow \B\times\tM$. Then, it is enough to prove that:
    $$(i')^* (p')_*([Z_w]) = \frac{\Pi_{wM}}{\Pi_{M\cap wM}} [Y_w]$$
    Using diagram (\ref{finite_Zw_as_bundle}), this follows from the Thom isomorphism. Then, injectivity of $\theta_0$ follows from Proposition \ref{tH_description}.
\end{proof}

\begin{proof}[Proof of Theorem \ref{finite_H(Z)=H_M}]
    We show that $\theta_0$ is the desired isomorphism. For all $w\in W$, define $A_w = \theta_0\left(\left[\overline{Z_w}\right]\right)$. These are elements of $\H_M$ by Lemma \ref{finite_Image_in_HM}. The definition of gluability and Lemma \ref{finite_leading_term_ZW} show that the elements $A_w$ satisfy the hypothesis of Proposition \ref{H_Sigma_basis}. This gives $\Im(\theta_0) = \H_M$. Lemma \ref{finite_leading_term_ZW} shows that $\theta_0$ is injective, so it induces an isomorphism $\HBM{G}{Z} \cong \H_M$.

    Combining Lemma \ref{finite_map_to_H0} and Proposition \ref{finite_H_0_thm}, it follows that $\theta_0$ intertwines the action of $\HBM{G}{Z}$ on $\HBM{G}{\B}\cong S.[\B]$ by convolution and the action of $\H_M$ on $S$.
\end{proof}

\section{The affine case}
\label{section_affine}

Let $G$ be as before. Let $\O \coloneqq \C[[\varpi]]$ be the ring of formal power series in $\varpi$ and let $\K = \C((\varpi))$ be the field of formal Laurent series. Let $G(\K)$ and $G(\O)$ be the formal loop group and the formal arc group associated to $G$. Let $I$ be the Iwahori subgroup of $G(\K)$ associated to $B$.

We consider two versions of $G(\K)$, in parallel. Let $\bbG\coloneqq G(\K) \rtimes \Crot$ or $G(\K)\rtimes \Crot \times \Cdil$. Here $\Crot \coloneqq \C^*$ acts on $G(\K)$ by loop rotations, and $\Cdil \coloneqq \C^*$ will act on (unions of) vector bundles by dilations along the fibers. Let $\bbB \coloneqq I \rtimes \Crot$ or $I \rtimes \Crot \times \Cdil$. We use notation from Section \ref{section_algebra} for each choice of $\bbG$. 

Explicitly, let $S = \C[\t][\hbar]$ or $\C[\t][\hbar, \epsilon]$ and let $Q$ be the subring of $\operatorname{Frac}(S)$ consisting of functions with no poles at $\hbar$, where $\hbar, \epsilon$ correspond to the identity characters $\Crot \rightarrow \C^*$ and $\Cdil \rightarrow \C^*$. Let $\tW = X_*(T) \rtimes W$ be the extended affine Weyl group associated to $G$ and let $\H$ be the subring of $Q[\tW]$ generated by $S$ and by Demazure operators.

The affine flag variety is the ind-scheme $\Fl \coloneqq G(\K) / I \cong \bbG/\bbB$, which is equipped with an action of $\bbB$. The orbits of this action are described by the Bruhat decomposition:

$$\Fl = \bigsqcup_{w\in \tW} \Fl_w$$

\noindent where $\Fl_w\coloneqq \bbB \dot{w} \bbB/\bbB$, for $\dot{w} \in \bbG$ being any lift of $w\in \tW$. The closures of the $\bbB$-orbits are given by:

$$\overline{\Fl_w} = \Fl_{\leq w} \coloneqq \bigsqcup_{y\leq w} \Fl_y$$

\noindent where $\leq$ is the Bruhat order on $\tW$.

Let $V$ be a representation of $\mathbb{G}$ and let $M$ be a closed $\bbB$-stable subspace of $V$. When $\bbG$ includes $\Cdil$, we additionally assume that $\Cdil$ acts on $V$ by vector space dilations, via the identity character $\Cdil \rightarrow \C^*$. Assume that each $\tT$-weight space of $V$ is finite dimensional. We also assume that, for all $w \in \tW$, the vector space 
$$M^w \coloneqq M / (M \cap \bigcap_{\{g\in \bbG \mid g\bbB \in \Fl_{\leq w}\}} g^{-1}M)$$ is finite dimensional.

Following \cite{BFN}, we define the variety of triples associated to the pair $(V,M)$ as follows. Consider the trivial vector bundle $\pi:\bbG/\bbB \times V \rightarrow \bbG/\bbB$, the trivial vector bundle $\bbG/\bbB \times M \hookrightarrow \bbG/\bbB \times V$ and the closed embedding $\tM \coloneqq \bbG \times^\bbB M \hookrightarrow \bbG/\bbB \times V$ given by $(g\bbB,m)\mapsto (g\bbB,g.m)$. We define the variety of triples $\R \coloneqq \tM \cap (\bbG/\bbB \times M) \subset \bbG/\bbB \times V$. This is an ind-scheme, equipped with an action of $\bbB$. For $w\in \tW$, let $\R_w \coloneqq \pi^{-1}(\Fl_w)$ and $\R_{\leq w} \coloneqq \pi^{-1}(\Fl_{\leq w})$. We have

\begin{equation}
\label{affine_Rw_description}
\begin{tikzcd}
	\R_w & \bbB \times^{\bbB\cap w \bbB w^{-1}} (M \cap wM)\\
    \\
    \Fl_w & \bbB/(\bbB\cap w \bbB w^{-1})
	\arrow["\cong"', from=1-1, to=1-2]
	\arrow["\pi", from=1-1, to=3-1]
	\arrow[""', from=1-2, to=3-2]
	\arrow["\cong", from=3-1, to=3-2]
\end{tikzcd}  
\end{equation}

Following \cite[\textsection 2(ii)]{BFN}, we define the $\bbB$-equivariant Borel-Moore homology of $\R$. A similar construction is explained in \cite{GKO}. Let $y, w\in \tW$ such that $w \leq y$. Consider the vector bundle $\tM^y \coloneqq \bbG\times^\bbB M^y \rightarrow \Fl$. The quotient map $M \twoheadrightarrow M^y$ induces a quotient map $\tM \twoheadrightarrow \tM^y$. Let $\tM^y_{\leq w}$ be the preimage of $\Fl_{\leq w}$ under $\tM^y \rightarrow \Fl$. Let $\R^y_{\leq w}$ be the image of the composition $\R_{\leq w} \hookrightarrow \tM_{\leq w} \twoheadrightarrow \tM^y_{\leq w}$. Then, $\R^y_{\leq w} \rightarrow \Fl_{\leq w}$ is a union of vector bundles of finite ranks and for $y' \geq y$ in $\tW$, the kernel of the quotient map $\R^{y'}_{\leq w} \twoheadrightarrow \R^y_{\leq w}$ is a vector bundle of rank $\dim M^{y'} - \dim M^y$ over $\Fl_{\leq w}$. For $i\geq 0$, let $\bbB_i$ be the image of $\bbB$ in $G(\O/\varpi^i \O) \rtimes \Crot$. As $\R^y_{\leq w}$ is of finite type, the action of $\bbB$ on it factors through $\bbB_i$ for $i$ large enough. Define

\begin{equation}
\label{affine_def_H(R_leq_w)}
\HBM{\bbB}{\R_{\leq w}} \coloneqq \operatorname{H}^{-\bullet}_{\bbB_i} (\R^y_{\leq w}, \omega_{\R^y_{\leq w}})[-2 \dim M^y]    
\end{equation}

\noindent As in \cite{BFN}, this definition is independent of $i$ and of $y\geq w$. Explicitly, for $y_1\geq y_2\geq w$ and $i$ large enough, we identify $\operatorname{H}^{-\bullet}_{\bbB_i} (\R^{y_1}_{\leq w}, \omega_{\R^{y_1}_{\leq w}})[-2 \dim M^{y_1}] \cong \operatorname{H}^{-\bullet}_{\bbB_i} (\R^{y_2}_{\leq w}, \omega_{\R^{y_2}_{\leq w}})[-2 \dim M^{y_2}]$ using the pullback along the smooth map $\R^{y_1}_{\leq w} \rightarrow \R^{y_2}_{\leq w}$, whose fibers are vector spaces of complex dimension $\dim M^{y_1} - \dim M^{y_2}$.

For $w_1, w_2, y \in \tW$ with $w_1 \leq w_2 \leq y$ and for $i$ large enough, pushforward along the closed embedding $\R^y_{\leq w_1} \hookrightarrow \R^y_{\leq w_2}$ induces a map $\HBM{\bbB}{\R_{\leq w_1}} \rightarrow \HBM{\bbB}{\R_{\leq w_2}}$. This map is well-defined and independent of $i$ and $y$. Define:

\begin{equation}
\label{affine_def_H(R)}
\HBM{\bbB}{\R} \coloneqq \varinjlim \HBM{\bbB}{\R_{\leq w}}
\end{equation}

\noindent where the direct limit is taken with respect to the Bruhat order in $\tW$. Note that $\HBM{\bbB}{\R}$ is a module over $H_\bbB^\bullet(\pt) \cong H_{\tT}^\bullet(\pt) \cong S$.

As in \cite[\textsection 3(iii)]{BFN}, we equip $\HBM{\bbB}{\R}$ with a convolution product, using the following diagram, which is the analogue of \cite[(3.2)]{BFN}.

\begin{equation}
\label{affine_diagram_R_product}
\begin{tikzcd}[column sep=large, row sep=large]
\mathcal{R} \times \mathcal{R} \arrow[d, "i \times \text{id}_{\mathcal{R}}"{left}] & p^{-1}(\mathcal{R} \times \mathcal{R}) \arrow[l, "\tilde{p}"{above}] \arrow[r, "\tilde{q}"] \arrow[d] & q(p^{-1}(\mathcal{R} \times \mathcal{R})) \arrow[r, "\tilde{m}"] \arrow[d] & \mathcal{R} \arrow[d, "i"] \\
\tM \times \mathcal{R} & 
\bbG \times \mathcal{R} \arrow[l, "p"{above}] \arrow[r, "q"{above}] & 
\bbG \times^\bbB \mathcal{R} \arrow[r, "m"{above}] & 
\tM,
\end{tikzcd}
\end{equation}

Here $i:\R \hookrightarrow \tM$ is the closed embedding given by the definition of $\R$, the top row consists of closed ind-subvarieties of the bottom row, and the maps in the bottom row are given by:

$$([g_1, g_2 s], [g_2, s]) \mapsfrom (g_1, [g_2, s]) \mapsto [g_1, [g_2, s]] \mapsto [g_1 g_2, s].$$

For $z_1, z_2 \in \HBM{\bbB}{\R}$, define their product $z_1 \cdot z_2 \coloneqq \tilde{m}_* (\tilde{q}^*)^{-1} p^* (z_1\boxtimes z_2)$, where the maps are defined as in \cite[(3.7), Remark 3.8]{BFN}. Theorem 3.10 in \cite{BFN} holds with the same proof, making $\HBM{\bbB}{\R}$ into an algebra.

Following \cite[Theorem 4.9, \textsection4.4.3]{HKW}, we define an $\HBM{\bbB}{\R}$-module structure on $\HBM{\bbB}{\pt}$. We use the following analogue of diagram (9) in \cite{HKW}, with $N(\O)$ replaced by the point~$\{0\}$:

\begin{equation}
\label{affine_diagram_0_module}
\begin{tikzcd}[column sep=large, row sep=large]
\mathcal{R} \times \{0\} \arrow[d, "i \times \text{id}"{left}] & \bbG \times \{ 0 \} \arrow[l, "\tilde{p_0}"{above}] \arrow[r, "q_0"] \arrow[d, "\text{id}"] & \bbG \times^{\bbB} \{ 0 \} \arrow[r, "m_0"] \arrow[d, "\text{id}"] & \{ 0 \} \arrow[d, "\text{id}"] \\
\tM \times \{ 0 \} & 
\bbG \times \{0\} \arrow[l, "p_0"{above}] \arrow[r, "q_0"{above}] & 
\bbG \times^\bbB \{ 0\} \arrow[r, "m_0"{above}] & 
\{ 0\},
\end{tikzcd}
\end{equation}

Define the action of $z\in \HBM{\bbB}{\R}$ on $c \in \HBM{\bbB}{0}$, by $z\star c \coloneqq (m_0)_* (q_0^*)^{-1}(\tilde{p_0})^*(z \boxtimes c)$. Theorem 4.9 and Proposition 4.13 in \cite{HKW} hold in our case with the same proofs, showing that this defines an $\HBM{\bbB}{\R}$-module structure on $\HBM{\bbB}{0} = \HBM{\bbB}{\pt}$. This is a free $S$-module of rank 1, generated by the fundamental class $[\pt] \in \HBM{\bbB}{\pt}$, as $\text{H}^\bullet_{\bbB}(\pt) \cong S$.

For a representation $\mathcal{N}$ of $T\times \Crot$ or $T\times \Crot\times \Cdil$, let $\Sigma(\mathcal{N})$ be the collection of weights of $\mathcal{N}$, counted with multiplicity. As in Section \ref{section_finite}, denote $\Pi_{\mathcal{N}} \coloneqq \Pi_{\Sigma(\mathcal{N})}$, $\H_M \coloneqq \H_{\Sigma(M)}$ and  $\Pi_{wM}\coloneqq  \Pi_{\Sigma(wM)}= \Pi_{w(\Sigma(M))}$ for $w\in \tW$. Our assumption on $M$ implies that, for all $w\in \tW$, $M/(M\cap wM)$ and $wM / (M \cap wM)$ are finite dimensional, so we can consider $\frac{\Pi_{M}}{\Pi_{M\cap wM}}$ and $\frac{\Pi_{wM}}{\Pi_{M\cap wM}}$ as elements of $S$.

\begin{definition}
\label{affine_def_gluable}
    A pair $(V,M)$ is called gluable if $\frac{\Pi_{M}}{\Pi_{M\cap wM}}$ and $\frac{\Pi_{wM}}{\Pi_{M\cap wM}}$ are coprime in $S$, for all $w\in \tW$.
\end{definition}

As in Section \ref{section_finite}, gluability of a pair $(V,M)$ depends not only on $M$, but on the ambient space $V$, in which the intersection $M\cap wM$ is taken.

\begin{remark}
    The definition of gluability depends on whether $\bbG$ includes $\Cdil$. Given a pair $(V,M)$, where $V$ is a representation of $G(\K) \rtimes \Crot \times \Cdil$, gluability of $(V,M)$ with respect to $G(\K) \rtimes \Crot \times \Cdil$ is a weaker assumption than gluability with respect to $G(\K) \rtimes \Crot$. A similar phenomenon is described in \cite[Corollaries 2.9, 2.10]{GW} in the case of Coulomb branches.
\end{remark}

The main result of this section is:

\begin{theorem}
    \label{affine_H(R)=H_M}
     Assume that the pair $(V,M)$ is gluable. Then there exists an algebra isomorphism $\HBM{\bbB}{\R} \cong \H_M$. This isomorphism intertwines the action of $\HBM{\bbB}{\R}$ on $\HBM{\bbB}{\pt}$ and the action of $\H_M$ on $S$.
\end{theorem}

If $G$ has a simply connected derived group, the analogous theorem holds when equivariant Borel-Moore homology is replaced with equivariant $K$-theory with coefficients in $\C$.

\begin{remark}
\label{affine_DAHA_remark}
Taking $V = \g(\K)$ and $M = \text{Lie}(I)$, we recover the poles and residues construction of the double affine Hecke algebra (DAHA) $\ddH$ and its trigonometric degeneration in \cite[Theorem 6.3.1]{GKV}, see also \cite{VV} and \cite[Theorem 6.1]{GKO}.

Note that $\mathcal{Q}[\tW]$ acts on $\mathcal{Q}$ by $\mathcal{Q}$-linearly extending the usual action of $\tW$ on $\mathcal{Q}$. An action of $\mathcal{Q}[\tW]$ on the free rank $1$ module generated by a symbol $\mathbf{m}$, denoted by $\mathcal{Q}.\mathbf{m}$, is defined in \cite[(7.7)]{BEG} by the formula:

$$[s_i] : g . \mathbf{m} \mapsto s_i(g) \cdot \frac{{q}^{-1} \cdot e^{\alpha_i/2} - {q} \cdot e^{-\alpha_i/2}}{{q}^{-1} \cdot e^{-\alpha_i/2} - {q} \cdot e^{\alpha_i/2}} . \mathbf{m}$$

The map $\mathbf{m}\mapsto e^{-\hat{\rho}} \varPi_M^{-1}$ induces an isomorphism of $\mathcal{Q}[\tW]$-modules $\mathcal{Q}.\mathbf{m} \cong \mathcal{Q}.\varPi_M^{-1}$. Under this isomorphism, Proposition \ref{H_Sigma_is_normalizer} says that $\ddH$ is the subalgebra of operators in $\mathcal{Q}[\tW]$ which preserve both $\S.\mathbf{m} \subset \mathcal{Q}.\mathbf{m}$ and $\S\subset \mathcal{Q}$.    
\end{remark}

\subsection*{Spherical case}

Let $\P = G(\O) \rtimes \Crot$ or $G(\O) \rtimes \Crot \times \Cdil$, so $\bbG / \P \cong \text{Gr}_{G}$, the affine Grassmannian of $G$. Let $V$ be a representation of $\bbG$ and let $M$ be a $\P$-stable subspace of it, satisfying the same conditions as before. Define the variety of triples $\R_\P \coloneqq (\bbG\times^\P M) \cap (\bbG/\P \times M)$ and its $\P$-equivariant Borel-Moore homology $\HBM{\P}{\R_\P}$ as before. Using a diagram similar to (\ref{affine_diagram_R_product}), we equip it with an algebra structure. 

Denote by $\Hsph \coloneqq e \H e$ and $\Hsph_M \coloneqq \Hsph  \cap \Pi_M^{-1} \Hsph \Pi_M \subset eQ[\tW]e$, where
$$e \coloneqq \frac{1}{|W|}\sum_{w\in W} [w] \in Q[\tW]$$
Gluability for the pair $(V,M)$ is equivalent to the condition that $\frac{\Pi_{M}}{\Pi_{M\cap e^\lambda M}}$ and $\frac{\Pi_{e^\lambda M}}{\Pi_{M\cap e^\lambda M}}$ are coprime in $S$, for all $e^\lambda \in X_*(T) \subset \tW$. 

\begin{corollary}
\label{affine_spherical_H(R)=H_M}
    If the pair $(V,M)$ is gluable, then there exists an isomorphism of algebras $\HBM{\P}{\R_\P} \cong \Hsph_M$.
\end{corollary}

\begin{proof}
    Following the proof of Corollary 6.3 in \cite{GKO}, we obtain an isomorphism of algebras $\HBM{\P}{\R_\P} \cong e \HBM{\bbB}{\R} e$. Using Theorem \ref{affine_H(R)=H_M} and Lemma \ref{H_Sigma_sph_description}, we obtain $\HBM{\P}{\R_\P} \cong \Hsph_M$.
\end{proof}

In particular, given a finite dimensional $G$-representation $N$, one can take $V = N(\K)$ and $M=N(\O)$. Then $\HBM{\P}{\R_\P}$ is the quantized Coulomb branch corresponding to $(G,N)$, and Corollary \ref{affine_spherical_H(R)=H_M} recovers a quantized version of Theorem 2.8 in \cite{GW}. A version of this result was originally proved in \cite[Theorem 1]{Tel}. Also, this corollary provides another description of the algebra described in \cite[Corollary 6.3]{GKO}.

\subsection*{Flavor symmetry}

It is possible to extend our results in the case of flavor symmetry, as follows. Let $G\subset \widetilde{G}$ be finite dimensional reductive groups, such that $G$ is normal in $\widetilde{G}$. Choose maximal tori $T \subset \widetilde{T}$. Let $\widetilde{\tT}\coloneqq \widetilde{T}\times \Crot$, $\widetilde{\th} \coloneqq \text{Lie}(\widetilde{\tT})$, $S=\C[\widetilde{\th}]$, $Q$ the ring of rational functions on $\widetilde{
\th}$ with no poles along imaginary root hyperplanes and $\tW$ be the extended affine Weyl group of $G$. As before, we define the (extended) nil-Hecke algebra $\H$ as the subalgebra of $Q[\tW]$ generated by $S$ and the Demazure operators corresponding to elements of $\tW$. For a collection $\Sigma \subset X^*(\widetilde{\tT})$ satisfying assumption (\ref{finiteness_assumption}), we define $\Pi_\Sigma$ and $\H_\Sigma \subset Q[\tW]$ as in section \ref{section_algebra}. The results of section \ref{section_algebra} hold with the same proofs.

Let $\bbB\subset \widetilde{\bbB}$ be Iwahori subgroups of $G(\O)\rtimes \Crot \subset \widetilde{G}(\O) \rtimes \Crot$. Assume that the action of $G(\K)$ on $V$ extends to an action of $\widetilde{G}(\K)$, and that $M$ is stable under the action of $\widetilde{\bbB}$. Define the variety of triples $\R$ as before. As in \cite[3(viii)]{BFN}, the action of $\bbB$ on $\R$ extends to an action of $\widetilde{\bbB}$. We define $\HBM{\widetilde{\bbB}}{\R}$ and equip it with an algebra structure using the same constructions as for $\HBM{{\bbB}}{\R}$. In this context, Definition \ref{affine_def_gluable} is the same, but $S=\C[\widetilde{\bbh}]$, and $\Pi_M$ is a (possibly infinite) product of characters of $\widetilde{\tT}$. Then, imitating the proof of Theorem \ref{affine_H(R)=H_M}, it follows that if the pair $(V,M)$ is gluable, then there exists an isomorphism $\HBM{\widetilde{\bbB}}{\R} \cong \H_M$, intertwining the two actions on $S$.

Using Theorem 4.2 in \cite{BEF}, it is likely that Theorem \ref{affine_H(R)=H_M}, and thus the results in Section \ref{section_algebra}, apply to cyclotomic double affine Hecke algebras, recovering Theorem 2.10 and Remark 3.12 in \cite{BEF}.

\subsection*{Proof of Theorem \ref{affine_H(R)=H_M}}

Our strategy of proof of Theorem \ref{affine_H(R)=H_M} is similar to that of Theorem \ref{finite_H(Z)=H_M}. As $\HBM{\bbB}{\R}$ is a direct limit of $\bbB$-equivariant Borel-Moore homology groups, it comes equipped with an action of $\text{H}^\bullet_\bbB(\pt) \cong S$. 

Define $\HBM{\bbB}{M}$ to be the free module of rank $1$ generated by the fundamental class $[M]$. By this, we mean that for all $y\in \tW$, the $S$-module $\HBM{\bbB}{M^y}$ is a free module of rank $1$ generated by the fundamental class $[M^y]$, and for $y_1\geq y_2$ in $\tW$, the pullback along the smooth map $M^{y_1} \rightarrow M^{y_2}$ maps $[M^{y_2}]$ to $[M^{y_1}]$.

Let $w\leq y$ in $\tW$. Using pushforward along the closed embedding $\{1\bbB\} \times M^y \hookrightarrow \R^y$ and the algebra structure on $\HBM{\bbB}{\R}$, we obtain an action of $\HBM{\bbB}{M^y}$ on $\HBM{\bbB}{\R^y_{\leq w}}$. This action is independent of the choice of $y$, and commutes with the maps $\HBM{\bbB}{\R^y_{\leq w_1}} \rightarrow \HBM{\bbB}{\R^y_{\leq w_2}}$. Thus, we obtain an action of $\HBM{\bbB}{M}$ on $\HBM{\bbB}{\R}$.

\begin{lemma}
    The action of $S\cong \text{H}^\bullet_{\bbB}(\pt)$ on $\HBM{\bbB}{\R}$ and the action of $S \cong S.[M] = \HBM{\bbB}{M}$ on $\HBM{\bbB}{\R}$ agree.
\end{lemma}

\begin{proof}
    The lemma is a consequence of the fact that multiplication in $\HBM{\bbB}{\R}$ is $S$-linear in the first variable. This follows as in \cite[Theorem 3.10]{BFN}.
\end{proof}

Thanks to this lemma, we view $\HBM{\bbB}{\R}$ as an $S$-module. 

Let $y, w\in \tW$ such that $y\geq w$. By diagram (\ref{affine_Rw_description}), $\R_w$ is a vector bundle over the $\bbB$-orbit $\Fl_w \subset\Fl_{\leq w}$. Consider its closure  $\overline{\R_w}$ in $\R_{\leq w}$. Similarly, we define $\overline{\R^y_w} \subset \R^y_{\leq w}$. For $y_1 \geq y_2 \geq w$, under the pullback along the smooth map $\R^{y_1}_{\leq w} \rightarrow \R^{y_2}_{\leq w}$, the fundamental class $[\overline{\R^{y_2}_w}]$ is mapped to $[\overline{\R^{y_1}_w}]$. Thus, using definitions (\ref{affine_def_H(R_leq_w)}) and (\ref{affine_def_H(R)}), we define $[\overline{\R_w}] \in \HBM{\bbB}{\R}$ independently of $y$. Arguing as in the proof of Lemma \ref{finite_S_mod_basis}, the fundamental classes $\left\{[\overline{\R_w}]\right\}_{w\in \tW}$ form a basis of the $S$-module $\HBM{\bbB}{\R}$.

Let $\rho_0 : \HBM{\bbB}{\R} \rightarrow \End(\HBM{\bbB}{\pt})$ be the representation described before using diagram (\ref{affine_diagram_0_module}). Using the same convolution
diagram for $V=M=0$, we obtain a representation $\rho : \HBM{\bbB}{\Fl} \rightarrow \End(\HBM{\bbB}{\pt})$.

Let $p_0:\R \hookrightarrow \tM$ be the closed embedding which was used to define $\R$ and let $i_0 : \Fl \hookrightarrow \tM$ be the $0$-section embedding. Let $\iota_0\coloneqq (i_0)^* \circ (p_0)_* : \HBM{\bbB}{\R} \rightarrow \HBM{\bbB}{\Fl}$. This is the analogue of the map $z^*$ in \cite[Lemma 5.11]{BFN}. Adapting the proofs in \cite[Proposition 4.15, \textsection 4.4.3]{HKW}, we obtain:

\begin{lemma}
\label{affine_0_map_to_H}
    The following diagram commutes
    \[\begin{tikzcd}
	\HBM{\bbB}{\R} \\
	&&  \End\left(\HBM{\bbB}{\pt}\right) \\
	\HBM{\bbB}{\Fl}
	\arrow["\rho_0", from=1-1, to=2-3]
	\arrow["\iota_0"', from=1-1, to=3-1]
	\arrow["\rho"', from=3-1, to=2-3]
\end{tikzcd}\]
\end{lemma}

We identify $\HBM{\bbB}{\pt}$ with $S$ and we denote by $\alpha:\End \left( \HBM{\bbB}{\pt}\right) \rightarrow \End (S)$ the resulting isomorphism. The next proposition follows from \cite{KK1}.

\begin{proposition}
\label{affine_H_0_thm}
The assignment
$$\left[ \Fl_{\leq w}\right] \mapsto D_w$$
extends to an $S$-linear algebra isomorphism $\theta : \HBM{\bbB}{\Fl} \xrightarrow{\cong}\H$, such that the following diagram commutes:
    \[\begin{tikzcd}
	\HBM{\bbB}{\Fl} && \End \left( \HBM{\bbB}{\pt}\right) \\
	\\
	\H && \End (S)
	\arrow["\rho"', from=1-1, to=1-3]
	\arrow["\theta", from=1-1, to=3-1]
	\arrow["\alpha", from=1-3, to=3-3]
	\arrow["\rho'",from=3-1, to=3-3]
\end{tikzcd}\]
    Here $\rho'$ is the action of $\H$ on $S$ described in Section \ref{section_algebra}.
\end{proposition}

Because the action of $\tW$ on $S$ is faithful, the action of $\H$ on $S$ is faithful. So $\rho$ is injective.

Following \cite[Theorem 4.9]{HKW}, we define an action of $\HBM{\bbB}{\R}$ on $\HBM{\bbB}{M}$, that is for every $w\leq y$ in $\tW$ we construct an action of $\HBM{\bbB}{\R^y_{\leq w}}$ on $\HBM{\bbB}{M^y}$, such that these actions commute with the pushforwards and smooth pullbacks used in definitions (\ref{affine_def_H(R_leq_w)}) and (\ref{affine_def_H(R)}). Consider the following version of diagram (9) in \cite{HKW}. It is an extension of our diagram (\ref{affine_diagram_0_module}).

\begin{equation}
\label{affine_diagram_M_module}
\begin{tikzcd}[column sep=large, row sep=large]
\mathcal{R} \times \{0\} \arrow[d] & \bbG \times \{ 0 \} \arrow[l, "\tilde{p}_0"{above}] \arrow[r, "q_0"] \arrow[d] & \bbG \times^{\bbB} \{ 0 \} \arrow[r, "m_0"] \arrow[d] & \{ 0 \} \arrow[d] \\
\mathcal{R} \times M \arrow[d, "i \times \text{id}"{left}] & p_M^{-1} (\R \times M) \arrow[l, "\tilde{p}_M"{above}] \arrow[r, "\tilde{q}_M"] \arrow[d] & \R \arrow[r, "\tilde{m}_M"] \arrow[d, "i"] & M \arrow[d] \\
\tM \times M & \bbG \times M \arrow[l, "p_M"{above}] \arrow[r, "q_M"{above}] & \tM \arrow[r, "m_M"{above}] & V,
\end{tikzcd}
\end{equation}

Here, all vertical arrows are closed embeddings, and the maps on the bottom row are given by:

$$([g, m], m) \mapsfrom (g, m) \mapsto [g, m] \mapsto gm$$

Using the second row of diagram (\ref{affine_diagram_M_module}), define the action of $z\in \HBM{\bbB}{\R}$ on $c \in \HBM{\bbB}{M}$, by:

\begin{equation}
\label{affine_H(R)_on_H(M)_action}
z\star c \coloneqq (\tilde{m}_M)_* (\tilde{q}_M^*)^{-1}({\tilde{p}_M^*})(z \boxtimes c)
\end{equation}

Theorem 4.9 in \cite{HKW} holds in our case with the same proof, showing that this defines an $\HBM{\bbB}{\R}$-module structure on $\HBM{\bbB}{M}$. Recall that we defined an $\HBM{\bbB}{\R}$-module structure on $\HBM{\bbB}{\pt}$ using the same composition of maps for the corresponding maps on the first row of diagram (\ref{affine_diagram_M_module}).

Let $\theta_0 \coloneqq \theta \circ\iota_0: \HBM{\bbB}{\R}\rightarrow \H$.

\begin{lemma}
\label{affine_Image_in_HM}
    $\Im(\theta_0) \subset \H_M$.
\end{lemma}

\begin{proof}
    Let $y\geq w$ in $\tW$. As in the proof of Lemma \ref{finite_Image_in_HM}, it is enough to prove that the following diagram commutes:
    \[\begin{tikzcd}
	\HBM{\bbB}{\R^y_{\leq w}} \otimes \HBM{\bbB}{0} && \HBM{\bbB}{0} \\
	\\
	\HBM{\bbB}{\R^y_{\leq w}} \otimes \HBM{\bbB}{M^y} && \HBM{\bbB}{M^y}&& 
	\arrow["\star"', from=1-1, to=1-3]
	\arrow["\text{id} \otimes \varepsilon_*", from=1-1, to=3-1]
	\arrow["\varepsilon_*", from=1-3, to=3-3]
	\arrow["\star",from=3-1, to=3-3]
\end{tikzcd}\]

Here, $\varepsilon : \{ 0 \}\hookrightarrow M^y$ is the closed embedding of the origin inside the vector space $M^y$, and the horizontal maps are the action maps defined before. The commutativity of the diagram follows from proper base change along the Cartesian squares in the first two rows of diagram (\ref{affine_diagram_M_module}).
\end{proof}

\begin{remark}
    Denote by $\rho_M : \HBM{\bbB}{\R} \rightarrow  \End \left( \HBM{\bbB}{M}\right)$ the action defined in (\ref{affine_H(R)_on_H(M)_action}). Let $p_M:\R \hookrightarrow \Fl\times M$ be the closed embedding which was used to define $\R$ and let $i_M : \Fl \hookrightarrow \Fl \times M$ be the $0$-section embedding. These maps are $\bbB$-equivariant, where $\Fl \times M$ is equipped with the diagonal $\bbB$-action. Let $\iota_M\coloneqq (i_M)^* \circ (p_M)_* : \HBM{\bbB}{\R} \rightarrow \HBM{\bbB}{\Fl}$. Following the proof of Lemma \ref{affine_0_map_to_H}, it can be shown that the following diagram commutes
    \[\begin{tikzcd}
	\HBM{\bbB}{\R}  &&   \End \left( \HBM{\bbB}{M}\right)\\
	\\
	\HBM{\bbB}{\Fl} &&  \End \left( \HBM{\bbB}{0}\right)
	\arrow["\rho_M"', from=1-1, to=1-3]
	\arrow["\iota_M", from=1-1, to=3-1]
	\arrow[from=1-3, to=3-3]
	\arrow["\rho",from=3-1, to=3-3]
\end{tikzcd}\]
Here, the right vertical arrow is the isomorphism induced by the Thom isomorphism, i.e. $\HBM{\bbB}{M} \cong \HBM{\bbB}{0}$. This is the affine analogue of Lemma 5.4.35 in \cite{CG}, while Lemma \ref{affine_0_map_to_H} is the affine analogue of Lemma \ref{finite_map_to_H0}. We obtain two embeddings $\iota_0, \iota_M : \HBM{\bbB}{\R} \rightarrow \HBM{\bbB}{\Fl}$. The proof of Lemma \ref{affine_Image_in_HM} shows that they differ by conjugation by $\Pi_M$ in $\HBM{\bbB}{\Fl} \cong \H$. Theorem \ref{affine_H(R)=H_M} shows that, if the pair $(V,M)$ is gluable, then 
$$\HBM{\bbB}{\R} \cong \H \times_{Q[\tW]} \H \coloneqq \{ (a,b) \in \H \times \H \mid  a = \Pi_M^{-1} b \Pi_M \} \cong \H \cap \Pi_M^{-1} \H \Pi_M$$

\noindent where the two embeddings $\H \hookrightarrow Q[\tW]$ are given by localization and its conjugate by $\Pi_M$, and the maps $\HBM{\bbB}{\R} \rightarrow \H$ are given by $\iota_0$ and $\iota_M$. This is an analogue of the morphism (2.6) and Theorem 2.8 in \cite{GW}.
\end{remark}

\begin{lemma}
\label{affine_leading_term_RW}
    The action of $\theta_0$ on fundamental classes is given by:
    $$\theta_0\left(\left[\overline{\R_w}\right]\right) = \frac{\Pi_{wM}}{\Pi_{M\cap wM}} D_w + \sum_{y< w} b_{w,y} D_y$$
    for some $b_{w,y} \in S$.
    In particular, $\theta_0$ is injective.
\end{lemma}

\begin{proof}
    After replacing $\overline{\R_w}$ with $\overline{\R^z_w}$ for $z\geq w$ in $\tW$, the proof is similar to that of Lemma \ref{finite_leading_term_ZW}.
\end{proof}

\begin{proof}[Proof of Theorem \ref{affine_H(R)=H_M}]
    The proof is similar to that of Theorem \ref{finite_H(Z)=H_M}, replacing the lemmas for a finite dimensional group with their affine analogues.
\end{proof}

\printbibliography

\end{document}